\documentclass{siamltex}
\usepackage{upquote}
\usepackage{graphicx}
\usepackage{algorithm}
\usepackage{algcompatible}
\usepackage{enumitem}
\usepackage{latexsym}
\usepackage{amssymb}
\usepackage{amsmath}
\usepackage{pdflscape}

\newcommand{\eqnref}[1]{(\ref{#1})}
\newtheorem{example}{Example}[section]
\usepackage{hyperref}
\newcommand{\eps}{\epsilon}
\usepackage{mathabx} % for \odiv
\graphicspath{{figures/}} 

\title{Instability of the Sherman-Morrison formula and stabilization by iterative refinement}
\author{Behnam Hashemi\thanks{School of Computing and Mathematical Sciences, University of Leicester, Leicester, LE1 7RH, UK. {\tt b.hashemi@le.ac.uk}.}
\and 
Yuji Nakatsukasa\thanks{Mathematical Institute, University of Oxford, Oxford, OX2 6GG, UK. Supported by EPSRC grants EP/Y010086/1 and EP/Y030990/1. {\tt nakatsukasa@maths.ox.ac.uk}.}
}

\begin{document}
\maketitle
\date{\today}

\begin{abstract} 
Owing to its simplicity and efficiency, the Sherman-Morrison (SM) formula has seen widespread use across various scientific and engineering applications for solving rank-one perturbed linear systems of the form $(A+uv^T)x = b$. Although the formula dates back at least to 1944, its numerical stability properties have remained an open question and continue to be a topic of current research. We analyze the backward stability of the SM, demonstrate its instability in a scenario increasingly common in scientific computing and address an open question posed by Nick Higham on the proportionality of the backward error bound to the condition number of $A$.\ We then incorporate fixed-precision iterative refinement into the SM framework reusing the previously computed decompositions and prove that, under reasonable assumptions, it achieves backward stability without sacrificing the efficiency of the SM formula. 
While our theory does not prove 
the SM formula with iterative refinement always outputs a backward stable solution, empirically it is observed to eventually produce a backward stable solution in all our numerical experiments. 
We conjecture that with iterative refinement, the SM formula yields a backward stable solution provided that $\kappa_2(A), \kappa_2(A+uv^T)$ are both bounded safely away from $\eps_M^{-1}$, where $\eps_M$ is the unit roundoff.

\end{abstract}
\begin{keywords} 
Sherman-Morrison formula; backward stability; iterative refinement. \end {keywords}

\section{Introduction}
The Sherman--Morrison--Woodbury (SMW) formula
\[
    (A+UV^T)^{-1}= A^{-1} - A^{-1}U(I + V^T\!A^{-1}U)^{-1}V^T\! A^{-1}
\]
has been used widely in scientific computing. Here $U,V$ are $n\times r$ with $r\leq n$, usually $r\ll n$. 
When $r=1$, it is often called the SM formula. Its primary usage is for solving linear systems of equations that undergo a low-rank update, as solving 
$(A+UV^T)x=b$ can be done efficiently if linear systems 
with respect to $A$ can be solved efficiently, e.g. when a factorization of $A$ is known or $A$ is sparse or has other structures that can be exploited for efficient solution. 
We refer the readers to \cite{Hager89} for an overview of the key ideas related to the SMW formula, historical developments \cite{Duncan44, Woodbury50, Bartlett51, SM50} and its many applications from ordinary and partial differential equations~\cite{Fortunato21, Olver13} to networks and optimization. It can also be extented to least-squares problems~\cite{guttel2024sherman}. 

Despite its widespread use, the numerical stability of SMW has remained an open problem. Yip's 1986 paper \cite{Yip86} remains the only reference to our knowledge to focus on its stability. Some more comments are in Hager \cite{Hager89} and Hao and Simoncini \cite{Hao21} etc. In these papers, the focus is on the conditioning of the matrix inversion $(I + V^T\!A^{-1}U)^{-1}$. Indeed, a poor choice of $U$ 
(note that $UV^T=UM^{-1}MV^T$ for any nonsingular $M$) results in an ill-conditioned $(I + V^T\!A^{-1}U)$, and it is easy to see that the resulting use of SMW gives poor results. 

Yip shows that $\kappa_2(I + V^T\!A^{-1}U)\leq \kappa_2(A)\kappa_2(A+UV^T)$
by choosing $U,V$ such that $U$ (or $V$) has orthonormal columns; this can be done simply by taking $M=R$ above, where $U=QR$ is the thin QR factorization. 
It follows that, if $A,A+UV^T$ are both well-conditioned, then it is easy to ensure the inversion of $I + V^T\!A^{-1}U$ is well-conditioned. Moreover, assuming $A^{-1}$ (or linear systems with $A$) are computed in a backward stable fashion in each step of the SMW formula application, we conclude that the forward error is small, and hence so is the backward error. That is, SMW can be implemented in a stable manner if both $A$ and $A+UV^T$ are well conditioned. 
In~\cite[Problem~26.2]{Higham02} Higham notes that examples can be found where SM is not backward stable, and that the backward error in practice appears to be bounded by $\mathcal{O}(\eps_M\kappa_\infty(A)$. 

Ma, Boutsikas, Ghadiri and Drineas \cite{Ma25} recently studied the stability of the SMW formula deriving both backward and forward error bounds. 
Their analysis focuses on the errors in two steps: computing the inverse of $A$, and the inverse of the capacitance matrix $Z = I + V^T A^{-1} U$ whose size is smaller than that of $A$. Essentially, they assume that all other operations are performed in exact arithmetic and formulate sufficient conditions imposing bounds on the amount of forward error in computing the aforementioned matrix inverses, and on $\|A\|$ and $\|U\| \|V\|$ and $\|Z\|$ so as to bound the forward error and backward error. Their work studies the SMW formula's stability in terms of matrix inversion, not when used for linear systems; this is a different problem, and the stability (or otherwise) in one problem does not immediately imply (in)stability of the other\footnote{See the appendix for a more detailed discussion.}. We also note that a classical work by Govaerts~\cite{Govaerts91} examines the backward error of a $(n+1)\times (n+1)$ bordered linear system 
$\begin{bmatrix}
A & u \\
v^T & -1
\end{bmatrix}
\begin{bmatrix}
x \\
\zeta
\end{bmatrix} =
\begin{bmatrix}
b \\
0
\end{bmatrix}$, which is related to the SM formula, and derives the backward error for this system. It is not entirely clear how the result can be translated into a backward error with respect to $A+uv^T$. We describe Govaerts' work in more detail in the appendix.

In this work we examine the backward stability of solving linear systems using the SM(W) formula.\ We first identify a source of instability. Perhaps contrary to intuition and most prior studies, 
it has nothing to do with the conditioning of $I + V^T\!A^{-1}U$; it is unstable even when $r=1$. We then focus on the $r=1$ case and carefully examine the backward error of solving 
\begin{equation}  \label{eq:maingoal}
(A+uv^T)x = b  
\end{equation}
 via the SM(W) formula, i.e., $x = A^{-1}b - A^{-1}u(1 + v^T\!A^{-1}u)^{-1}v^T\! A^{-1}b$. For definiteness we present the implementation in Algorithm~\ref{SM:alg}. 
\begin{algorithm}[!h]
\caption{SM: Sherman-Morrison formula for solving $(A+uv^T)x = b$. It is assumed that linear systems with $A$ can be solved efficiently.}
\label{SM:alg}
\begin{algorithmic}[1]
\STATE Solve $Ay = b$ for $y$ to obtain $\hat y$. 
\STATE Solve $Az=u$ for $z$ to obtain $\hat z$.  
\STATE Compute $\alpha = v^T \hat y$ and $\beta = 1+v^T \hat z$, and let $\theta = \alpha / \beta$. 
\STATE Output $\hat x = \hat y - \theta \hat z$.
\end{algorithmic}
\end{algorithm}

We show that under some conditions, a question raised by Nick Higham~\cite[p.~570]{Higham02} can be answered in the affirmative, namely that the backward error of the SM formula can be bounded by $\eps_M \kappa(A)$.

All these point towards a cautionary note: the SM(W) formula, while very attractive for computational efficiency, can be dangerously unstable. 
Fortunately, we suggest a simple remedy to the instability: iterative refinement (IR). 
IR is a classical idea in NLA but has witnessed a resurged interest, partly due to the growing interest and availability of mixed precision arithmetic~\cite{carson2018accelerating}, but perhaps equally importantly, as a tool to correct or improve a fast but unstable algorithm~\cite{epperly2024fast}.

Despite the widespread use of SM(W) and IR, the combination appears to not have been used in the literature. The key contribution of this paper is the introduction and analysis of iterative refinement (IR) for the SM formula.\ With IR we find a correction to a (possibly inaccurate) solution $\hat x$ obtained by SM as follows: compute the residual $r=b-(A+uv^T)\hat x$, and solve $(A+uv^T)\delta x = r$ for the correction $\delta x$ using SM again, and output $x+\delta x$ as the solution. 
The resulting algorithm SM-IR is displayed in Algorithm~\ref{SMIR:alg}. 
It is worth noting that SM-IR has the same complexity as SM; it only doubles the cost. 
If necessary, one can repeat IR more than once, which can further improve the accuracy. 

\begin{algorithm}[!h]
\caption{SM-IR: Sherman-Morrison formula with iterative refinement for solving $(A+uv^T)x = b$.}
\label{SMIR:alg}
\begin{algorithmic}[1]
\STATE Compute $\hat x$ using SM.
\STATE Compute $r  = b - A \hat x - (v^T \hat x) u$.
\STATE Solve $A y_r = r$ for $y_r$ to obtain $\hat y_r$
\STATE Compute $\theta_r := \alpha_r / \beta$ where $\alpha_r = v^T y_r$ and $\beta$ same as in Step 1.
\STATE Output $\hat w = \hat x + y_r - \theta_r z$.
\STATE If the accuracy is insufficient with $\hat w$, set $\hat x:=\hat w$ and return to Step 2.
\end{algorithmic}
\end{algorithm}

It is worth noting that IR is often used together with mixed precision. For example, the classical analysis and usage of IR assumed higher-precision arithmetic is used for the computation of the residual~\cite[Ch.~12]{Higham02}. More recently, Carson and Higham~\cite{carson2018accelerating} study IR using multiple precisions. 
In this work, we focus on the simplest IR where fixed precision (e.g. double precision) is used throughout; mixed precision, while possibly helpful, is not observed to be necessary empirically, and our analysis assumes the same precision is used throughout.\ More generally, IR is often able to improve the accuracy of an unstable solution, even in fixed precision~\cite{jankowski1977iterative,skeel1980iterative}. 

We analyze the stability of SM-IR with one IR step, and show that under some mild conditions that can be verified on the fly, the algorithm outputs a backward stable solution $\hat w$, that is, one that satisfies 
$(A+uv^T+\Delta A)\hat w = b$ for some $\|\Delta A\|=\mathcal{O}(\eps_M\|A+uv^T\|)$. 
That is, while the SM formula is fundamentally unstable, the instability can be fixed with IR under mild assumptions. 

However, the assumptions are not always satisfied, and we are currently unable to definitively prove that SM-IR will eventually (after a modest number of IR steps) produce a backward stable solution for any linear system. Despite this, our numerical experiments suggest strongly that this is the case. 
We hence conjecture that SM-IR is backward stable under very mild conditions, namely that both $A$ and $A+uv^T$ have condition number bounded by $\eps_M^{-1}$. 
Another open problem is to deal with a higher-rank case $r>1$, that is, the SMW formula. One can clearly perform iterative refinement here, resulting in SMW-IR. An important open problem is to address the stability of SMW-IR. 

\paragraph{Notation}
Following Higham~\cite{Higham02}, quantities computed in floating-point arithmetic are denoted with a hat.
To avoid confusing notation (to reserve $u$ for the vector in the updated matrix $A+uv^T$), we denote unit roundoff by $\eps_M$. Usually in stability analysis we suppress terms that grow like a modest polynomial in the matrix dimensions, so the choice of matrix norm is unimportant; for this reason we use $\|\cdot\|$ in most cases. 
Capital letters denote matrices, and vectors are represented by $w,x,y$ and $z$. Other lower-case letters denote scalars. When the norm choice matters or improves clarity, we use $\|\cdot\|_2$ for the spectral norm of a matrix. $\kappa(A)=\|A\|\|A^{-1}\|$ denotes the condition number of $A$.

\section{A source of instability of the SMW formula}\label{sec:source}
Here we identify a situation where the SMW formula would naturally lead to a large error, and is expected to be unstable. 

To simplify the discussion here we assume $\|A+UV^T\|_2=1$ and $\|b\|_2=1$ without loss of generality, and also that $\|A\|_2=\mathcal{O}(1)$, which is a nontrivial assumption. Here are the remaining assumptions: 
\begin{enumerate}
\item $A$ is \emph{Ill-conditioned}: $\kappa_2(A)\gg 1$. We make no assumptions on $\kappa_2(A+UV^T)$. 
\item \emph{Small-norm solution}: The linear system $(A+UV^T)x=b$ has a solution $\|x\|_2=\mathcal{O}(1)$. By contrast, $Ay=b$ has solution $\|y\| = \mathcal{O}(\kappa_2(A))\gg 1$. 
\end{enumerate}
The second assumption merits further explanation. 
It is essentially an assumption on the right-hand side $b$: if we write $b=\sum_{i=1}^nc_iv_i$ where $A+UV^T=\sum_{i=1}^n\sigma_iu_iv_i^T$ is the singular value decomposition, then the assumption is requiring that $|c_i| = \mathcal{O}(\sigma_i)$. We refer to this as a small-norm assumption, and $x$ is a small-norm solution. We note that such situations are increasingly common in scientific computing: (i) in inverse problems, this is a standard phenomenon, and such assumptions are commonly made in the analysis, and called the \emph{Picard condition}~\cite[\S 1.2.3]{hansen2010discrete}. 
(ii) In numerical computation with \emph{frames} \cite{adcock2019frames}, it is an assumption made (and satisfied) in order to establish convergence and good numerical behavior of a numerically computed solution for approximation of functions in a domain. 
(iii) Indirectly related: the prominence of stochastic gradient descent comes (at least partially) from its ability to find a small-norm approximate solution to an (often highly underdetermined) optimization problem~\cite{hastie2022surprises}. 

More generally, one could argue that huge-norm solutions to ill-conditioned problems are unlikely to have practical/physical relevance; for example given a solution $x$ with $\|x\|=10^{16}$, the task of computing $x^Ty$ for some $y$ is difficult to do with accuracy better than $\mathcal{O}(1)$ in double-precision arithmetic. 

For these reasons we regard the assumption (ii), while certainly a nontrivial one, as quite natural and an important case in practice. In experiments, one can form such $b$ by simply taking a randomly generated (e.g. Gaussian) vector $x$ and setting $b=(A+UV^T)x$. By contrast, if $b$ is taken randomly, one gets the opposite situation\footnote{These statements hold with (exponentially) high probability.} where $\|x\|_2=\mathcal{O}(\kappa_2(A+UV^T))$. Gill, Wright, Murray~\cite[p.~102]{gill2021numerical} describe how such small-norm solutions should arise naturally in practical problems. 

Let us explain where the instability comes from. 
Recall that 
$x = (A+UV^T)^{-1}b = A^{-1}b - A^{-1}U(I + V^T\!A^{-1}U)^{-1}V^T\! A^{-1}b=y-\theta z$. In any implementation of the SMW formula (including Algorithm~\ref{SM:alg}), one would first compute $y$, then find $\theta, z$ (using $y$). 
Then assuming a backward stable solver is used, with the computed $\hat y$ we have $(A+\Delta A)\hat y=b$, where $\|\Delta A\|/\|A\|=\epsilon$. By Neumann series this implies 
$\hat y - y = A^{-1}\Delta A A^{-1}b + \mathcal{O}(\|\Delta A\|^2)$. 
We next compute $\theta$ and $z$. Even if this is done exactly or with high relative accuracy (which in general we do not expect), the result will have error at least 
$
\|\hat x-x\|\gtrsim 
\|\hat y - y\| = \|A^{-1}\Delta A A^{-1}b\| =
\|A^{-1}\Delta A y\|= \mathcal{O}(\epsilon \kappa_2(A)\|y\|) = \mathcal{O}(\epsilon (\kappa_2(A))^2)$, by the assumption $\|y\| = \mathcal{O}(\kappa_2(A))$. 

It remains to see that this is indeed not a backward stable solution. 
Note that by definition a backward stable solution $\hat x$ needs to satisfy 
$(A+UV^T+\epsilon)\hat x = b$, so the residual is 
$\|(A+UV^T)\hat x -b\| = \|\epsilon\hat x\|
\leq  \|\epsilon\| \|\tilde x\|= \mathcal{O}(\eps_M)$, by the assumption that $\|x\|=\mathcal{O}(1)$; which implies  also that $\|\hat x\|=\mathcal{O}(1)$, as $\|(A+UV^T+\epsilon)(x-\hat x)\|= \|\epsilon x\| = \mathcal{O}(\eps_M)$, and so $\|x-\hat x\|\leq \mathcal{O}(\eps_M/\sigma_{\min}(A+UV^T+\epsilon))=\mathcal{O}(1)$; Here we are tacitly assuming $\eps_M\kappa_2(A)<1$.

So if $\hat x$ is backward stable we need $\|\hat x\|=\mathcal{O}(1)$; however, $\hat x$ is computed as the difference $\hat y-\hat z$, where $\|\hat y\|_2=\mathcal{O}(\kappa_2(A))$, so simply by representing $\hat y$ we incur error $\mathcal{O}(\eps_M\kappa_2(A))$. 
Thus $x-\hat x$ is at least $\mathcal{O}(\eps_M\kappa_2(A))$, and so the residual is 
\begin{align*}
\|(A+UV^T)\hat x-b\| &= 
\|(A+UV^T)(\hat x-x)\|
\gtrsim \eps_M\|A+UV^T\|\kappa_2(A)
=\mathcal{O}(\eps_M\kappa_2(A)), 
\end{align*}
where the estimate assumes that the error $\hat x-x$ is unstructured with respect to the singular vector expansion of $A+UV^T$. Note that we also derive an upper bound for the residual in~\eqnref{res_bnd_new:eq}, which (under the current assumption that $\kappa_2(A)\gg \kappa_2(A+UV^T)$ in addition to further assumptions set out there) shows that $\mathcal{O}(\eps_M\kappa_2(A))$ is a tight estimate of the residual; which also matches Higham's conjecture that the SM backward error is bounded from above by $\eps_M \kappa_2(A)$. Thus the backward error of $\hat x$  as a solution for $(A+UV^T)x=b$ is at least $\eps_M\kappa_2(A+UV^T)$, and so it is not a backward stable solution. 

\section{Higham's open problem: Bounding backward error in SM}
In the remainder of the paper we restrict ourselves to the $r=1$ case and consider the stability of the SM (not SMW) formula for solving $(A+uv^T)x=b$.

An open problem in Higham's magnum opus on numerical stability~\cite[p. 570]{Higham02} is whether SM has a backward error bound proportional to $\kappa_{\infty}(A) = \|A^{-1}\|_{\infty} \|A\|_{\infty}$. 

Our plan in what follows is to first derive a neat representation for the SM residual, bound it, and then invoke the Rigal-Gaches formula~\cite{Rigal67} for the backward error~\cite[Lem.~1.1]{Higham02} 
\begin{equation}
\label{berr_formula:eq}
\eta(\hat x) = \frac{\|r\|}{\|B\| \|\hat x\| + \|b\|}.
\end{equation}

To bound the residual, we consider the solution obtained by SM as
\begin{equation}
\label{SM_sol_rep_new:eq}
\hat x := fl(y - \frac{\alpha}{\beta} z) = \hat y \ominus \frac{\alpha}{\beta} \hat z
\end{equation}
where $\alpha = v^T \hat y$, and $\beta = 1+v^T \hat z$. Here, a backward stable solver is used to compute $\hat y$ and $\hat z$ as solutions to $Ay = b$ and $Az=u$, respectively. Therefore, 
\begin{equation}
\label{berr_y:eq}
(A+\Delta_1) \hat y = b, \quad \mbox{ where } \| \Delta_1\| \leq c_1 \eps_M \| A\|
\end{equation}
where $c_1$ is a modest multiple of $1$; see \cite[p. 171]{Watkins} or \cite[Thm. 16.2 ]{TB:2022}. Similarly,
\begin{equation}
\label{berr_z:eq}
(A+\Delta_2) \hat z = u, \quad \mbox{ where } \| \Delta_2\| \leq c_2 \eps_M \| A\|.
\end{equation}
We know from the standard floating point arithmetic model that for each $1\leq i\leq n$ we have $\hat y_i \ominus \frac{\alpha}{\beta} \hat z_i = (\hat y_i - \frac{\alpha}{\beta} \hat z_i) (1+\delta_i)$ where $|\delta_i| \leq \eps_M$. So, the above representation gives\footnote{Here again we are  focusing on the error in the subtraction and assume that the errors in $\hat \alpha, \hat \beta$, their division and multiplication with $\hat z_i$ are insignificant.}
\begin{equation}
\label{SM_sol_rep_new2:eq}
\hat x = (\hat y - \frac{\alpha}{\beta}\hat z) + \delta \hat x
\end{equation}
where $\delta \hat x$ is the vector whose entries are $(\hat y_i - \frac{\alpha}{\beta} \hat z_i)  \delta_i$, i.e.,
\begin{equation}
\label{cancelation_err1:eq}
|\delta \hat x| \leq \eps_M |\hat y - \frac{\alpha}{\beta}\hat z|
\end{equation}
With this we now turn to $r$. We have
\begin{align*}
r &= b - B \hat x = b - (A+uv^T) (\hat y \ominus \frac{\alpha}{\beta} \hat z)  =  b - (A+uv^T) \big( (\hat y - \frac{\alpha}{\beta}\hat z) + \delta \hat x \big)\\
&= b - A \hat y + \frac{\alpha}{\beta} A \hat z - u v^T \hat y + \frac{\alpha}{\beta} (v^T \hat z) u - (A+uv^T) \delta \hat x \\
&= b - A \hat y + \frac{\alpha}{\beta} \bigg( A \hat z + (v^T \hat z) u \bigg) - u v^T \hat y - (A+uv^T) \delta \hat x  \quad \ \ \ \ \mbox{and \eqnref{berr_z:eq} yields }  A \hat z = u - \Delta_2 \hat z  \\
&= \Delta_1 \hat y + \frac{\alpha}{\beta} \bigg( u -  \Delta_2 \hat z + (v^T \hat z) u \bigg) - u v^T \hat y - (A+uv^T) \delta \hat x \quad \mbox{where with \eqnref{berr_y:eq} }  b - A \hat y = \Delta_1 \hat y \\
&= \Delta_1 \hat y - \frac{\alpha}{\beta} \big(\Delta_2 \hat z \big) + 
u \frac{\alpha}{\beta} \bigg( \underbrace{1 + (v^T \hat z)}_{\beta} \bigg) - u v^T \hat y - (A+uv^T) \delta \hat x  \\
&= \Delta_1 \hat y - \frac{\alpha}{\beta} \big(\Delta_2 \hat z \big) + 
u \alpha  - u v^T \hat y - (A+uv^T) \delta \hat x \quad \quad \mbox{(as $u \alpha  - u v^T \hat y = 0$)}\\
&= \Delta_1 \hat y - \frac{\alpha}{\beta} \big(\Delta_2 \hat z \big) - (A+uv^T) \delta \hat x
\end{align*}
Using the bounds in \eqnref{berr_y:eq}, \eqnref{berr_z:eq} and \eqnref{cancelation_err1:eq} yields
\begin{align}
\| r \| &\leq \|\Delta_1 \| \|\hat y\| +  |\frac{\alpha}{\beta}| \| \Delta_2\| \|\hat z\| + \| A+uv^T\| \|\delta \hat x\| \nonumber\\
& \leq c \ \eps_M \| A\| \Big( \|\hat y\| +  |\frac{\alpha}{\beta}| \|\hat z\| \Big) + \eps_M \| A+uv^T\|  \|\hat y - \frac{\alpha}{\beta}\hat z\| \nonumber\\
& \leq c \ \eps_M \| A\| \Big( \|\hat y\| +  |\frac{\alpha}{\beta}| \|\hat z\| \Big) + \eps_M \| A+uv^T\|  \big( \|\hat y\| + |\frac{\alpha}{\beta}| \|\hat z\| \big) \nonumber \\
& = \eps_M (c\ \| A\| + \| A + uv^T\|) \Big( \|\hat y\| +  |\frac{\alpha}{\beta}| \|\hat z\| \Big) \label{res_bnd:eq}
\end{align}
where $c = \max \{c_1, c_2 \}$ is again a modest multiple of $1$. 

It turns out that the last term $\|\hat y\| +  |\frac{\alpha}{\beta}| \|\hat z\|$ plays a crucial role in the stability analysis of SM. It represents the (absolute) condition number of the subtraction of $\frac{\alpha}{\beta} \hat z$ from  $\hat y$ when forming $\hat x$. We now seek a bound on this condition number as it tells us where cancellation may occur in SM. So, we now focus on finding upper bounds for $\hat y$ and $|\frac{\alpha}{\beta}| \|\hat z\|$. It is worth noting that these quantities are easy to compute on the fly; in particular if $|\frac{\alpha}{\beta}| \|\hat z\|$ is nicely bounded by $\max(\|A^{-1}\|_2,\|(A+uv^T)^{-1}\|_2)$, then it follows that the residual is bounded by $\|b\|_2$, that is, the SM formula gives a solution that at least has a residual smaller than that of $x=0$. One can then employ iterative refinement (possibly more than one step) to reduce the residual until we arrive at a backward stable solution. Since $\|\hat z\|$ depends on the right-hand side the argument is not a trivial matter of repeating the same analysis. We study iterative refinement in detail in the next section.

\begin{lemma} \label{help_bnd:lem}
Assume that the errors in $\alpha, \beta$, their division and multiplication with $\hat z_i$ are insignificant. If 
\begin{eqnarray}
& |v^T \hat z| > 1.1 \label{hypoth1:eq}
\end{eqnarray}
then 
\begin{equation}
\label{lemmaBnd:eq}
\|\hat y\| +  |\frac{\alpha}{\beta}| \|\hat z\| \leq \check c\|\hat y\|
\end{equation}
where $\check c$ is a constant. 
\end{lemma}

\begin{proof}
We start by examining the size of $\frac{1}{\beta} = \frac{1}{1+v^T \hat z}$ where $\hat z = fl(A^{-1} u)$. We have
\[
\frac{1}{1+v^T \hat z} = \frac{1}{v^T \hat z \big( (v^T \hat z)^{-1} + 1 \big)} = \frac{1}{v^T \hat z} \frac{1}{1 + \zeta}
\]
where from \eqnref{hypoth1:eq}, it follows that $\zeta := (v^T \hat z)^{-1}$ has a magnitude less than 1 hence $\frac{1}{1 + \zeta} = 1 - \zeta + \zeta^2 - \dots$ yields
\[
\frac{1}{1+v^T \hat z} = \frac{1}{v^T \hat z} \big( 1 - \frac{1}{v^T \hat z} + \frac{1}{(v^T \hat z)^2} - \dots \big)
= \frac{1}{v^T \hat z} - \frac{1}{(v^T \hat z)^2} + \frac{1}{(v^T \hat z)^3}  - \dots
\]
and so
\begin{equation}
\label{alpha_beta1:eq}
\frac{\alpha}{\beta} = \frac{v^T \hat y}{1+v^T \hat z} = \frac{v^T \hat y}{v^T \hat z} - \frac{v^T \hat y}{(v^T \hat z)^2} + \frac{v^T \hat y}{(v^T \hat z)^3}  - \dots
\end{equation}
Thus, 
\begin{align}
|\frac{\alpha}{\beta}| & \leq 
\frac{|v^T \hat y|}{|v^T \hat z |}\left(1+|\zeta|+|\zeta|^2+\cdots \right) \nonumber \\
&=\frac{1}{1-|\zeta|}\frac{\|v\|\ \|\hat y\| \ |\cos (\theta_{v,\hat y})| }{\|v\|\ \|\hat z \| \ |\cos (\theta_{v,\hat z})|} = \frac{1}{1-|\zeta|}\frac{\|\hat y\| \ |\cos (\theta_{v,\hat y})| }{\|\hat z \| \ |\cos (\theta_{v,\hat z})|} \label{alphaBetaBound:eq}
\end{align}
where $\theta_{v,\hat y}$ and $\theta_{v,\hat z}$ denote the angle between $v$ and $\hat y$ and between $v$ and $\hat z$, respectively. 
Multiplying both sides by $\|\hat z\| > 0$ gives
\begin{align*}
 |\frac{\alpha}{\beta}| \ \|\hat z\| & \leq
\frac{1}{1-|\zeta|}
 \frac{|\cos (\theta_{v,\hat y})| }{|\cos (\theta_{v,\hat z})|} \|\hat y\|. 
\end{align*}
We thus obtain~\eqref{lemmaBnd:eq} with $\check c=1+\frac{1}{1-|\zeta|}\frac{|\cos (\theta_{v,\hat y})| }{|\cos (\theta_{v,\hat z})|}$.
\end{proof}

The following result is an answer to Higham's question on the proportionality of the SM backward error with the condition number of $A$.
\begin{proposition}
Assume $Ay=b$ is solved with a normwise backward stable algorithm. Under the assumptions of the preceding lemma, we have
\begin{equation}
\label{res_bnd_new:eq}
\| r \| \leq \eps_M  \frac{\check c}{1 - c_1 \eps_M \kappa(A)}  \Big( c\ \| A\| + \| A + uv^T\| \Big) \| A^{-1}\| \|b\| + \mathcal{O}(\eps_M^2) 
\end{equation}
where $c_1, c$ are modest multiples of $1$, $c_1 \eps_M \kappa(A) < 1$, and $\check c$ is as in Lemma~\ref{help_bnd:lem}.
\end{proposition}
\begin{proof}
The previous lemma together with \eqnref{res_bnd:eq} give 
\[
\| r \| \leq \eps_M  \check c  (c\ \| A\| + \| A + uv^T\|) \|\hat y\| + \mathcal{O}(\eps_M^2) 
\]
In addition, from \eqnref{berr_y:eq}, $c_1 \eps_M \kappa(A) < 1$ and standard perturbation theory \cite[p. 54]{Stewart98} we can show that\footnote{Let $\delta y := y - \hat y$ denote the forward error in $\hat y$. We have
\[
A\ \delta y = A \hat y - A y  = (b - \Delta_1 \hat y) - Ay = - \Delta_1 \hat y \quad \to \delta y = -A^{-1} \Delta_1\ \hat y \to
\| \delta y \| \leq \|A^{-1}\| \| \Delta_1\|  \|\hat y\|
\]
and so from \eqnref{berr_y:eq} we get 
\[
\| \delta y \| \leq c_1\ \eps_M \kappa(A) \|\hat y\|. \qquad \mbox{ forward error in terms of } \|\hat y\|
\] Now, $\| \hat y \| = \|y + \delta y\| \leq \|y\| + \|\delta y\| \leq \|y \| + c_1\ \eps_M \kappa(A) \|\hat y\|$ yields
$\| \hat y \| \leq \frac{\| y\| }{1 - c_1 \eps_M \kappa(A)} = \frac{\| A^{-1} b\| }{1 - c_1 \eps_M \kappa(A)}.$
}
\begin{equation}
\label{temp_ineq2:eq}
\| \hat y \| \leq \frac{\| A^{-1}\| }{1 - c_1 \eps_M \kappa(A)} \| b\| \\
\end{equation}
from which \eqnref{res_bnd_new:eq} follows.
\end{proof}

This result shows that the backward error of the SM formula is bounded by $\mathcal{O}(\max(\kappa_2(A),\kappa_2(A+uv^T))\eps_M)$, if the assumptions hold. We emphasize that the assumptions, while verifiable on the fly, are nontrivial and not always true. It therefore remains an open problem to address whether the backward error of SM can always be bounded by $\mathcal{O}(\max(\kappa_2(A),\kappa_2(A+uv^T))\eps_M)$ or $\mathcal{O}(\kappa_2(A)\eps_M)$. 

\section{SM with iterative refinement} 
We now examine the stability of SM-IR as in Algorithm~\ref{SMIR:alg}. The main goal of this section is to prove Theorem \ref{SM_resBndCmp:thm}, which claims that the residual in SM admits the required form for invoking Theorem 12.3 in Higham's book~\cite{Higham02}. This can be used to establish that, under natural conditions, iterative refinement enhances backward stability of SM.

Invoking Theorem 12.3 requires a {\em componentwise} bound on the residual which is more intricate than the {\em normwise} bound we established in the previous section. Moreover, here we account for {\em all} rounding errors without exception, even though they do not affect the overall outcome.

The table below outlines all steps in both exact and floating-point arithmetic, along with the relevant bounds. Below and in what follows, given a vector $v\in\mathbb{R}^n$, $|v|\in\mathbb{R}^n$ takes the absolute values elementwise.
\begin{table}[!h]
\begin{tabular}{lll}
exact arithmetic op. & floating-point op. & backward error bound\\
1. $y = A\backslash b$ & $\hat y = fl(y)$ & $(A+\Delta_1) \hat y = b$, where $|\Delta_1| \leq \tilde \gamma_{n^2} e e^T |A|$\\
&  & \\
2. $z = A \backslash u$ & $\hat z = fl(z)$ & $(A+\Delta_2) \hat z = u$, where $|\Delta_2| \leq \tilde \gamma_{n^2} e e^T |A|$\\
&  & \\
3. $\alpha = v^T y$ & $\hat \alpha = fl(\alpha) = v^T \otimes \hat y$ & $\hat \alpha = v^T \hat y + \delta_{\hat y}$ with $|\delta_{\hat y}| \leq \gamma_n |v|^T |\hat y|$\\
&  & \\
4. $\beta = 1 + v^T z$ & $\hat \beta = fl(\beta) = 1 \oplus (v^T \otimes \hat z)$ & $\hat \beta = 1 + (v^T \hat z + \delta_{\hat z}) + \delta_{+}$ with \\
&  & $|\delta_{\hat z}| \leq \gamma_n |v|^T |\hat z|$, $|\delta_{+}| \leq \eps_M (1 + |v^T \otimes \hat z|) $\\
& & and $v^T \otimes \hat z = v^T \hat z + \delta_{\hat z}$\\
&  & \\
5. $\theta = \alpha / \beta$ & $\hat \theta = fl(\theta) =  \hat \alpha \odiv \hat \beta$  & $\hat \theta = \frac{\hat \alpha}{\hat \beta} + \delta_{\hat \theta}$ with 
$|\delta_{\hat \theta}| \leq \frac{\eps_M |\hat \alpha|}{|\hat \beta|} $ \\
&  & \\
6. $w = \theta z$ & $\hat w = fl(w) =  \hat \theta \otimes \hat z$  & $\hat w = \hat \theta \hat z + \delta_{\times}$ with $|\delta_{\times}| \leq \eps_M |\hat \theta| |\hat z|$\\
&  & \\
	                  
7. $x = y - w$ & $\hat x = fl(x) =  \hat y \ominus \hat w$ & $\hat x = \hat y - \hat w + \delta_{-}$ with \\
  		                  & & \ \ \ \  \ \ \ \ \  $|\delta_-| \leq \eps_M |\hat y - \hat w| \leq \eps_M |\hat y| + \eps_M |\hat w|$\\
& & \qquad \qquad \quad $\leq \eps_M |\hat y| + \eps_M |\hat \theta| |\hat z| + \eps_M^2 |\hat \theta| |\hat z|$
\end{tabular}
\end{table}

It is assumed that all steps are performed using backward-stable algorithms. Specifically, the bounds in Steps 1-2 hold for solving the systems with QR factorization computed using Householder or Givens transformations \cite[pp. 361-368]{Higham02}. We now begin bounding the residual
\begin{align}
r &= b - (A+uv^T) \hat x = b - A\hat x - (v^T \hat x) u \nonumber \\
& = b - A\hat y + A \hat w - A \delta_{-} - (v^T \hat x) u \quad \mbox{where } A \hat w = A (\hat \theta \hat z + \delta_{\times}) = \hat \theta A \hat z + A \delta_{\times} \nonumber \\
& = b - ( b - \Delta_1 \hat y) + A (\hat \theta \hat z + \delta_{\times}) - A \delta_{-} - (v^T \hat x) u \nonumber \\
& =  \Delta_1 \hat y - \hat \theta \Delta_2 \hat z  + A \delta_{\times} - A \delta_{-} + u (\hat \theta - v^T \hat x) \label{res_temp:eq}
\end{align}
The only term in the last expression that does not explicitly contain a delta-type factor of order $\eps_M$ is the last one. Thus, we focus on bounding $\hat \theta - v^T \hat x$, whose exact-arithmetic counterpart, $\theta - v^T x$, is zero.
\begin{lemma} \label{thetaZ:lem}
\begin{align*}
|\hat \theta - v^T \hat x| & \leq \gamma_{n+1} |v|^T |\hat y| +
 \gamma_{n+4} \frac{|\hat \alpha|}{|\hat \beta|} |v|^T |\hat z| + 2 \eps_M \frac{|\hat \alpha|}{|\hat \beta|}  + \mathcal{O}(\eps_M^2)
\end{align*}
\end{lemma}
\begin{proof}
Step 7 yields
\[
v^T \hat x = v^T \hat y - v^T \hat w + v^T  \delta_{-} = v^T \hat y - v^T (\hat \theta \hat z + \delta_{\times}) + v^T  \delta_{-} = v^T \hat y - \hat \theta v^T \hat z - v^T \delta_{\times} + v^T  \delta_{-}.
\]
Hence
\[
\hat \theta - v^T \hat x = \hat \theta - v^T \hat y + \hat \theta v^T \hat z + v^T \delta_{\times} - v^T  \delta_{-} = (1 + v^T \hat z) \hat \theta - v^T \hat y + v^T (\delta_{\times} -  \delta_{-}). 
\]
Substituting $\hat \theta = \frac{\hat \alpha}{\hat \beta} + \delta_{\hat \theta}$ from Step 5 gives
\begin{equation}
\label{thetaZ:eq}
\hat \theta - v^T \hat x = (1 + v^T \hat z) \frac{\hat \alpha}{\hat \beta} - v^T \hat y + (1 + v^T \hat z) \delta_{\hat \theta} + v^T (\delta_{\times} -  \delta_{-}). 
\end{equation}

On the other hand, from Step 4 we have $\hat \beta = 1 + v^T \hat z + \delta_{\hat z} + \delta_{+}$. Substituting $1 + v^T \hat z = \hat \beta -  \delta_{\hat z} - \delta_{+}$ into \eqnref{thetaZ:eq} yields
\begin{align}
\hat \theta - v^T \hat x &= (\hat \beta -  \delta_{\hat z} - \delta_{+}) \frac{\hat \alpha}{\hat \beta} - v^T \hat y + (1 + v^T \hat z) \delta_{\hat \theta} + v^T (\delta_{\times} -  \delta_{-}) \nonumber \\
& = \hat \alpha - v^T \hat y - \frac{\hat \alpha}{\hat \beta} \delta_{\hat z}  - \frac{\hat \alpha}{\hat \beta} \delta_{+} + (1 + v^T \hat z) \delta_{\hat \theta} + v^T (\delta_{\times} -  \delta_{-}) \nonumber \\
& = \delta_{\hat y} - \frac{\hat \alpha}{\hat \beta} \delta_{\hat z}  - \frac{\hat \alpha}{\hat \beta} \delta_{+} + (1 + v^T \hat z) \delta_{\hat \theta} + v^T (\delta_{\times} -  \delta_{-}) \label{theta_minus_vx:eq}
\end{align}
which implies
\begin{align}
|\hat \theta - v^T \hat x| & \leq |\delta_{\hat y}| + |\frac{\hat \alpha}{\hat \beta}| |\delta_{\hat z}|  + |\frac{\hat \alpha}{\hat \beta}| |\delta_{+}| + (1 + |v|^T |\hat z|) |\delta_{\hat \theta}| + |v|^T (|\delta_{\times}| + |\delta_{-}|) \nonumber \\
& =  |\delta_{\hat y}| + |\frac{\hat \alpha}{\hat \beta}| \big( |\delta_{\hat z}|  + |\delta_{+}| \big) + (1 + |v|^T |\hat z|) \eps_M \frac{|\hat \alpha|}{|\hat \beta|} + |v|^T (|\delta_{\times}| + |\delta_{-}|) \nonumber\\
& =  |\delta_{\hat y}| + |\frac{\hat \alpha}{\hat \beta}| \Big( |\delta_{\hat z}|  + |\delta_{+}| + (1 + |v|^T |\hat z|) \eps_M \Big) + |v|^T (|\delta_{\times}| + |\delta_{-}|) 
\label{thetaZ2:eq}
\end{align}
Looking at the last term, recall from Steps 6 and 7 that
\[
|\delta_{\times}| \leq \eps_M |\hat \theta| |\hat z| \qquad \mbox{ and } \qquad
|\delta_-| \leq \eps_M |\hat y| + \eps_M |\hat \theta| |\hat z| + \mathcal{O}(\eps_M^2)
\]
which give
\[
|\delta_{\times}| + |\delta_-| \leq 2 \eps_M |\hat \theta| |\hat z|  + \eps_M |\hat y| +  \mathcal{O}(\eps_M^2).
\]
Next, from Step 5 we have
\begin{equation} \label{thetaHatBnd:eq}
|\hat \theta|  \leq  |\frac{\hat \alpha}{\hat \beta}| + |\delta_{\hat \theta}| 
\leq  |\frac{\hat \alpha}{\hat \beta}| + \eps_M  |\frac{\hat \alpha}{\hat \beta}|
\end{equation}
which yields
\begin{equation} \label{twoDelta:eq}
|\delta_{\times}| + |\delta_-| \leq 2 \eps_M |\frac{\hat \alpha}{\hat \beta}| |\hat z|  + \eps_M |\hat y| +  \mathcal{O}(\eps_M^2).
\end{equation}
Substituting this into \eqnref{thetaZ2:eq} gives
\[
|\hat \theta - v^T \hat x| \leq
 |\delta_{\hat y}| + |\frac{\hat \alpha}{\hat \beta}| \Big( |\delta_{\hat z}|  + |\delta_{+}| + (1 + |v|^T |\hat z|) \eps_M + 2 \eps_M |v|^T |\hat z| \Big) + \eps_M |v|^T |\hat y| +  \mathcal{O}(\eps_M^2)
\]
Also we know from Step 4 that
\begin{align*}
|\delta_{+}|  &\leq \eps_M (1 + |v|^T |\hat z| + |\delta_{\hat z}|) 
\leq \eps_M (1 + |v|^T |\hat z| + \gamma_n |v|^T |\hat z|) = \eps_M (1 + |v|^T |\hat z|) + \mathcal{O}(\eps_M^2)
\end{align*}
which gives
\begin{align*}
|\hat \theta - v^T \hat x| & \leq
 |\delta_{\hat y}| + |\frac{\hat \alpha}{\hat \beta}| \Big( |\delta_{\hat z}|  + 2 \eps_M (1 + |v|^T |\hat z|) + 2 \eps_M |v|^T |\hat z| \Big) + \eps_M |v|^T |\hat y| +  \mathcal{O}(\eps_M^2)\\
 & \leq  \gamma_n |v|^T |\hat y| + \eps_M |v|^T |\hat y| + |\frac{\hat \alpha}{\hat \beta}| \Big( \gamma_n |v|^T |\hat z|  + 2 \eps_M (1 + |v|^T |\hat z|) + 2 \eps_M |v|^T |\hat z| \Big)  +  \mathcal{O}(\eps_M^2)\\
 & =  (\gamma_n + \eps_M) |v|^T |\hat y| + |\frac{\hat \alpha}{\hat \beta}| \Big( (\gamma_n + 4 \eps_M) |v|^T |\hat z|  + 2 \eps_M \Big)  +  \mathcal{O}(\eps_M^2)
\end{align*}
Since $(\gamma_n + \eps_M) \leq \gamma_{n+1}$ and  $(\gamma_n + 4\eps_M) \leq \gamma_{n+4}$, the proof is complete.
\end{proof}

To obtain a bound on the residual, we require the following two auxiliary results.
\begin{lemma}
\begin{equation} \label{yHatBnd:eq}
|\hat y| \leq (1 + \eps_M) |\hat x| + |\frac{\hat \alpha}{\hat \beta}| (1 + 3\eps_M) |\hat z| + \mathcal{O}(\eps_M^2).
\end{equation}
\end{lemma}
\begin{proof}
Applying the reverse triangle inequality to $\hat x = \hat y - \hat w + \delta_{-}$ gives
$|\hat y| \leq  |\hat x| + |\hat w| + |\delta_-| $. From  $\hat w = \hat \theta \hat z + \delta_{\times}$ with $|\delta_{\times}| \leq \eps_M |\hat \theta| |\hat z|$ and \eqnref{thetaHatBnd:eq}
we obtain
\[
|\hat w| \leq |\hat \theta| |\hat z| (1 + \eps_M)
  \leq \Big( |\frac{\hat \alpha}{\hat \beta}| (1 + \eps_M) + \eps_M  |\frac{\hat \alpha}{\hat \beta}| \Big) |\hat z| + \mathcal{O}(\eps_M^2).
\]
Also, $|\delta_-| \leq \eps_M |\hat y| + \eps_M |\frac{\hat \alpha}{\hat \beta}|  |\hat z| + \mathcal{O}(\eps_M^2)$ leading to
\[
|\hat y| \leq |\hat x| + |\frac{\hat \alpha}{\hat \beta}| (1 + 3\eps_M) |\hat z| + \eps_M |\hat y|  + \mathcal{O}(\eps_M^2).
\]
Applying $(1-\eps_M)^{-1} = 1 + \eps_M + \mathcal{O}(\eps_M^2)$ yields the result.
\end{proof}

\begin{lemma}
\begin{equation} \label{zHatBnd:eq}
|\hat z| \leq |A^{-1} u| + \tilde \gamma_{n^2} |A^{-1}| e e^T |A| |A^{-1} u| +  \mathcal{O}(\eps_M^2).
\end{equation}
\end{lemma}
\begin{proof}
From Step 2 we have
\[
\hat z = (A+\Delta_2)^{-1} u =  \big(A (I+ A^{-1}\Delta_2) \big)^{-1} u =: (I+ F) A^{-1} u 
\] 
where $|F| \leq \tilde \gamma_{n^2} |A^{-1}| e e^T |A| + \mathcal{O}(\eps_M^2)$.
\end{proof}

Our main result is the following.
\begin{theorem} \label{SM_resBndCmp:thm}
The SM residual satisfies 
\begin{equation} \label{SM_resBndCmp:eq}
|b - A\hat x - (v^T \hat x) u| \leq \eps_M \ \big( g(A,u,v) |\hat x| + h(A, u, v, b) \big)
\end{equation}
 where 
\[
g(A,u,v) = (c n^2 e e^T + I) |A| + (n+1) |u| |v|^T
\]
and
\begin{equation} \label{h_formula:eq}
h(A,u,v, b) = \frac{|\hat \alpha|}{|\hat \beta|} \Big( (2d n^2 e e^T + 3I) |A| + 2 (n+4) |u| |v|^T\Big) |A^{-1} u| + 2 \frac{|\hat \alpha|}{|\hat \beta|} |u|.
\end{equation}
in which $c$ and $d$ are small integer constants\footnote{appearing in the notation $\tilde \gamma_{k} = \gamma_{ck}  = \frac{c k \eps_M}{1 - c k \eps_M}$.}.
\end{theorem}
\begin{proof}
Beginning with \eqnref{res_temp:eq},
\[
|r|  \leq  |\Delta_1| |\hat y| + |\hat \theta| |\Delta_2| |\hat z|  + |A| (|\delta_{\times}| +|\delta_{-}|) + |u| |\hat \theta - v^T \hat x|
\]
we apply Lemma \ref{thetaZ:lem} to obtain
\begin{align*}
|r|  & \leq  |\Delta_1| |\hat y| + |\hat \theta| |\Delta_2| |\hat z|  + |A| (|\delta_{\times}| +|\delta_{-}|)  \\
& \qquad + \gamma_{n+1} |u| |v|^T |\hat y| +
 \gamma_{n+4} \frac{|\hat \alpha|}{|\hat \beta|}  |u| |v|^T |\hat z| + 2 \eps_M \frac{|\hat \alpha|}{|\hat \beta|} |u|  + \mathcal{O}(\eps_M^2).
\end{align*}
Applying \eqnref{thetaHatBnd:eq} and \eqnref{twoDelta:eq} yield
\begin{align*}
|r|  & \leq  |\Delta_1| |\hat y| + \Big( |\frac{\hat \alpha}{\hat \beta}| + \eps_M  |\frac{\hat \alpha}{\hat \beta}| \Big) |\Delta_2| |\hat z|  \\
& \qquad + |A| \Big( 2 \eps_M |\frac{\hat \alpha}{\hat \beta}| |\hat z|  + \eps_M |\hat y| +  \mathcal{O}(\eps_M^2) \Big) \\
& \qquad + \gamma_{n+1} |u| |v|^T |\hat y| +
 \gamma_{n+4} \frac{|\hat \alpha|}{|\hat \beta|}  |u| |v|^T |\hat z| + 2 \eps_M \frac{|\hat \alpha|}{|\hat \beta|} |u|  + \mathcal{O}(\eps_M^2) \\
 & = \Big( |\Delta_1| + \eps_M |A| + \gamma_{n+1} |u| |v|^T \Big) |\hat y| \\
 & \qquad + \frac{|\hat \alpha|}{|\hat \beta|} \Big( |\Delta_2| + 2 \eps_M |A|
+  \gamma_{n+4}  |u| |v|^T \Big) |\hat z| + 2 \eps_M \frac{|\hat \alpha|}{|\hat \beta|} |u|  + \mathcal{O}(\eps_M^2) \\
 & \leq \Big( (\tilde \gamma_{n^2} e e^T + \eps_M I) |A| + \gamma_{n+1} |u| |v|^T \Big) |\hat y| \\
 & \qquad +  \frac{|\hat \alpha|}{|\hat \beta|} \Big( (\tilde \gamma_{n^2} e e^T + 2\eps_M I
) |A| +  \gamma_{n+4}  |u| |v|^T \Big) |\hat z| + 2 \eps_M \frac{|\hat \alpha|}{|\hat \beta|} |u|  + \mathcal{O}(\eps_M^2).
\end{align*}
Next, we apply \eqnref{yHatBnd:eq} and \eqnref{zHatBnd:eq}, pushing higher-order terms to $\mathcal{O}(\eps_M^2)$ to obtain
\begin{align*}
|r| & \leq \Big( (\tilde \gamma_{n^2} e e^T + \eps_M I) |A| + \gamma_{n+1} |u| |v|^T \Big) |\hat x| \\
 & +  \frac{|\hat \alpha|}{|\hat \beta|} \Big( (2\tilde \gamma_{n^2} e e^T + 3\eps_M I
) |A| +  (\gamma_{n+1}  + \gamma_{n+4})  |u| |v|^T \Big) |A^{-1} u| 
+ 2 \eps_M \frac{|\hat \alpha|}{|\hat \beta|} |u|  + \mathcal{O}(\eps_M^2).
\end{align*}
Note that $\gamma_{n+1}  + \gamma_{n+4} \leq 2\gamma_{n+4}$.
\end{proof}

In addition to Theorem~\ref{SM_resBndCmp:thm}, we require a bound on the difference between the residual $r$ and the computed residual $\hat r$ in SM to be able to apply \cite[Thm. 12.3]{Higham02}.

\begin{lemma} Let $r := b - A\hat x - (v^T \hat x) u$ and $\hat r = fl(r)$. Then, we have
\begin{equation} \label{rHatBnd:eq}
|r - \hat r| \leq \eps_M \ t(A,u,v,\hat x, b)
\end{equation}
where
\begin{equation} \label{t_rep:eq}
t(A,u,v,\hat x, b) = \frac{\gamma_{n+2}}{\eps_M} \Big(  |b| + (|A|  + |u| |v|^T)|\hat x| \Big).
\end{equation}
\end{lemma}
\begin{proof}
First, consider $p = b - A \hat x$ and $\hat p = fl(p)$. It is easy to see that $\hat p = p + \delta p$ where $|\delta p| \leq \gamma_{n+1} (|b| + |A| |\hat x|)$. Next, let $q = (v^T \hat x) u$ and $\hat q = fl(q) = (v^T \otimes \hat x) \otimes u$. We know that $\hat q = q + \delta q$ with (see \cite[p. 74]{Higham02} for instance)
\[
|\delta q| \leq (\gamma_{n}  + \eps_M (1 + \gamma_n)) |u| |v|^T |\hat x| \leq \gamma_{n+1}  |u| |v|^T |\hat x| + \mathcal{O}(\eps_M^2).
\]
Finally, we have $\hat r = \hat p \ominus \hat q$ which satisfies $\hat r = \hat p - \hat q + \check \delta_{-}$
with 
\[
|\check \delta_{-}| \leq \eps_M (|\hat p| + |\hat q|) \leq \eps_M (|p| + |\delta p|) + \eps_M (|q| + |\delta q|) \leq \eps_M (|b| + |A|\hat x|) + \eps_M  |u| |v|^T |\hat x| + \mathcal{O}(\eps_M^2).
\]
 We have 
\[
\hat r = (p + \delta p) - (q + \delta q) + \check \delta_{-} = (p - q) +  \delta p - \delta q + \check \delta_{-} = r + \delta r
\]
with $\delta r := \delta p - \delta q + \check \delta_{-}$ which yields
\begin{align*}
|r - \hat r| & = |\delta r |  \leq |\delta p| + |\delta q| + |\check \delta_{-}|\\
               & \leq (\gamma_{n+1} + \eps_M) (|b| + |A| |\hat x|) + (\gamma_{n+1} + \eps_M) |u| |v|^T |\hat x| + \mathcal{O}(\eps_M^2) \\
               & \leq \gamma_{n+2} (|b| + |A| |\hat x|) +\gamma_{n+2}  |u| |v|^T |\hat x| + \mathcal{O}(\eps_M^2).
\end{align*}
\end{proof}

We are now ready to invoke \cite[Thm. 12.3]{Higham02} to establish our main result demonstrating that each SM iteration improves the stability of SM. See Alg.\ref{SMIR:alg}. 
\begin{theorem}
Let $A$ and $B := A + u v^T$ be nonsingular $n \times n$ matrices. Suppose the linear system $B x = b$ is solved using SM in floating-point arithmetic following one step of iterative refinement as in Algorithm~\ref{SMIR:alg}, where $y$, $z$ and $y_r$ are computed with a backward stable algorithm. Assume that the computed SM solution $\hat x$ satisfies~\eqnref{SM_resBndCmp:eq}, and the computed residual satisfies~\eqnref{rHatBnd:eq}. Then, the corrected solution $\hat w$ satisfies
\begin{equation} \label{SMIR_resBndCmp:eq}
|b - A\hat w - (v^T \hat w) u| \leq \eps_M \ \Big( h(A,u,v, \hat r) + t(A,u,v, \hat w, b)+ \big( |A|+|u||v|^T \big) |\hat w|\Big) + \eps_M \ q,
\end{equation}
where $q = \mathcal{O}(\eps_M)$ if $t(A,u,v,\hat x, b) - t(A,u,v,\hat w, b) = \mathcal{O}(\|\hat x - \hat w\|_{\infty})$.
\end{theorem}

Next we adapt Higham's explanation of \cite[Thm. 12.3]{Higham02} to our linear system $Bx=b$. Theorem~\ref{SMIR_resBndCmp:eq} indicates that, to first order, the componentwise relative backward error 
\[
\mbox{Berr}_{|B|,|b|}(\hat w) = \max_{i} \frac{|r|_i}{(|B| |\hat w| + |b|)_i}
\]
(see~\cite{Oettli64} and~\cite[Thm. 7.3]{Higham02})
will be small after one step of iterative refinement as long as $t(A,u,v,\hat w, b)$ and $h(A,u,v,\hat r)$ are bounded by a modest scalar multiple of $|B| |\hat w| + |b|$, the denominator of $\mbox{Berr}_{|B|,|b|}(\hat w)$. Following~\eqnref{t_rep:eq} we have
\[
t(A,u,v,\hat w, b) = \frac{\gamma_{n+2}}{\eps_M} \Big(  |b| + (|A|  + |u| |v|^T)|\hat w| \Big)
\] 
which is exactly what is required. In the case of $h$, note first that $h(A,u,v,\hat r)$ is obtained by replacing $b$ in \eqnref{h_formula:eq} with $\hat r$. This means that $\hat \alpha$ in \eqnref{h_formula:eq} is replaced with $\hat \alpha_r$, the quantity computed in Step 4 of Alg.~\ref{SMIR:alg}. We have therefore shown that

\begin{theorem} \label{oneIRenough:thm}
If 
\begin{equation} \label{h_r_formula:eq}
h(A,u,v, \hat r) = \Big( (2d n^2 e e^T + 3I) |A| + 2 (n+4) |u| |v|^T\Big) \frac{|\hat \alpha_r|}{|\hat \beta|}  |A^{-1} u| + 2 \frac{|\hat \alpha_r|}{|\hat \beta|} |u|
\end{equation}
is bounded by a modest scalar multiple of $|A+uv^T| |\hat w| + |b|$, then one step of SM-IR guarantees (componentwise relative) backward stability.
\end{theorem}

This is a general theorem; to be able to say anything more specific we would require knowledge of the size of 
\[
\frac{|\hat \alpha_r|}{|\hat \beta|}  |A^{-1} u|
\]
 against $\hat w$. In trying gain some insight, let us consider the two cases of large- and small-norm solutions. 
 \begin{itemize}
 \item If the solution has a large norm $\mathcal{O}(\sigma_{\min}(A+uv^T)\|b\|)$, it is likely that SM is already backward stable in which case IR is not required but even if it is applied, assuming the SM-IR solution $\hat w$ has a similar norm to the SM solution $\hat x$ as well as the exact solution $x$, we can expect $\frac{|\hat \alpha_r|}{|\hat \beta|}  \|A^{-1} u\|$ to be bounded by $\|\hat w\|$ hence backward stability is maintained. 
  \item On the other hand, consider the small-norm scenario in which $A$ and $A+uv^T$ are ill-conditioned, $\|x\| = \mathcal{O}(1)$ (e.g., $\|A\| = \|b\| = \mathcal{O}(1)$), but $\|y\|$ and $\|z\|$ are $\mathcal{O}(\kappa(A)) \gg 1$. In a nutshell, Lemma~\ref{help_bnd:lem} says $\frac{\hat \alpha}{\hat \beta} = \mathcal{O}(1)$, as follows from \eqnref{alphaBetaBound:eq} and assumptions such as $\frac{\cos (\theta_{v,\hat y})}{\cos (\theta_{v,\hat z})} = \mathcal{O}(1)$, and that the SM residual $\hat r$ before any IR steps is proportional to $\eps_M \kappa(A)$; similarly, assuming Lemma~\ref{help_bnd:lem} applies when $\hat y_r$ is used instead of $\hat y$, we can argue that 
 \[
 \frac{|\hat \alpha_r|}{|\hat \beta|} \|\hat z \| \leq \check c \| \hat{y}_{\hat r}\| + \mathcal{O}{(\eps_M)} 
 \]
where $\hat{y}_{\hat r} = fl(y_{\hat r}) = fl(A^{-1} \hat r)$ whose norm could be bounded analogous with~\eqnref{temp_ineq2:eq} where $b$ is replaced with $\hat r$. Assuming $\hat w = \mathcal{O}(1)$, we have 
 \begin{align*}
\frac{|\hat \alpha_r|}{|\hat \beta|} \max_i \frac{|A^{-1} u|_i}{|\hat w|_i} &\approx 
 \frac{\|A^{-1} \hat r\|}{\|A^{-1} u\|} \|A^{-1}u\|= \|A^{-1} \hat r\| \leq \|A^{-1}\| \| \hat r\| \\ & = \mathcal{O}(\kappa(A))\ \mathcal{O}(\eps_M \kappa(A)) = \mathcal{O}(\eps_M \kappa^2(A)).
 \end{align*}
Roughly speaking, this implies that in the small-norm solution scenario, for SM to achieve backward stability with one step of fixed-precision IR, we would require $\kappa(A) < \eps_M^{-1/2}$, which corresponds to $6.7 \times 10^7$ in double precision. In such cases the original SM residual $r$ for the SM solution $\hat x$ to $Bx=b$ (i.e., before any IR steps) has a norm bounded by $\eps_M \kappa(A) \|b\| \approx  10^{-8} \|b\|$. With the first IR step, as discussed above, the deciding quantity is essentially $\|A^{-1} \hat r\|$ which corresponds to the solution norm of a linear system $Ax = \hat r $. This can be bounded by $\eps_M \kappa(A) \| \hat r \| \lessapprox \eps_M^2 \kappa^2(A) \|b\| \lessapprox \eps_M \|b\|$ thereby achieving backward stability.
\end{itemize}

\section{Numerical experiments}\label{sec:exp}
In our experiments we compare the following algorithms for solving \eqnref{eq:maingoal}:
\begin{enumerate}
\item MATLAB backslash applied directly to the sum $B:= A+uv^T$: This is GEPP (Gaussian elimination with partial pivoting) based on computing the LU decomposition with partial pivoting. 
\item SM-LU: Algorithm~\ref{SM:alg} where the $A$-solves ($Ay=b$ and $Az=u$) are done using GEPP. 
\item SM-QR: Algorithm~\ref{SM:alg} using the QR factorization of $A$ for $A$-solves. 
\item SM-LU-IR: This is SM-IR (Algorithm~\ref{SMIR:alg}) using GEPP for $A$-solves. 
\end{enumerate}
The complexity of the first algorithm is cubic $\mathcal{O}(n^3)$ in the matrix dimension, $A\in\mathbb{R}^{n\times n}$. Assuming the factors of an LU or QR decomposition of $A$ are already available and the number of IR steps is $\mathcal{O}(1)$, the complexity of the rest of the methods is quadratic $\mathcal{O}(n^2)$. 

While GEPP/LU is known to have adversarial examples where $Ax=b$ would not be solved in a stable fashion, such matrices are known to be extremely rare~\cite[Ch.~22]{TB:2022}, and empirically LU performs exceptionally well (QR, by contrast, is always backward stable). For this reason we do not present SM-QR-IR, as (expectedly) its empirical performance is almost identical to SM-LU-IR.

In our experiments, the normwise relative backward error of an approximate solution $\tilde x$ is computed using the bound in \eqnref{berr_formula:eq}. We take $\max_{i=1}^n \frac{ |r_i|}{\big( |B| |\tilde x| + |b| \big)_i}$ as the componentwise relative backward error. In SM-LU-IR we repeated IR until the relative backward error fell below $5\eps_M\approx 5\times 10^{-16}$.

Based on the insight obtained in Section~\ref{sec:source}, and because the empirical performance differ based on the parameters, we split the experiments in terms of the magnitude of $\kappa(A),\kappa(A+uv^T)$, and whether or not the norm of the solution $x$ is small $\mathcal{O}(1)$ or large $\mathcal{O}(1/\sigma_{\min}(A+uv^T)) = \mathcal{O}(\kappa(A+uv^T))$. 
Throughout this section, $\kappa$ denotes the 2-norm condition number $\kappa(A)=\|A\|_2\|A^{-1}\|_2$. All reported backward and forward errors are relative. We consider four cases based on whether or not $\kappa(A),\kappa(A+uv^T)\gg 1$, and when $\kappa(A+uv^T)\gg 1$, we further split into two cases: (i) the solution $x$ has small norm $\|x\|=O(1)$, and (ii) large norm $\|x\|=O(\kappa(A+uv^T))$.

\subsection{Case 1(i): both $\mathbf{\kappa(A), \kappa(A+uv^T) \gg 1}$, small-norm solution}

\begin{example} \label{ill_ill_smallNorm_sparse:ex} \normalfont
We generate sparse random matrices $A$ of size $n = 8000$ using the \texttt{sprandn} routine in MATLAB with condition numbers of $10^6, 10^8, 10^{10}$ and $10^{12}$. The density of $A$ is set to $0.0001$ resulting in a matrix with about $8000$ nonzero entries. In this example we had $\kappa(A) = 10^6$. The vectors $u$ and $v$ are generated with \texttt{randn}. In all four matrices, the condition number of $B=A+uv^T$ was about $10^{10}$; all of these are reported on the horizontal axes in Figure~\ref{fig:ill_moderate_smallNorm_sparse}. We take the exact solution to also be a \texttt{randn} vector for which $\|x\| \approx 90$ and then form $b:= Bx$. 

\begin{figure}[!t]
\center
\includegraphics[width=0.7\textwidth]{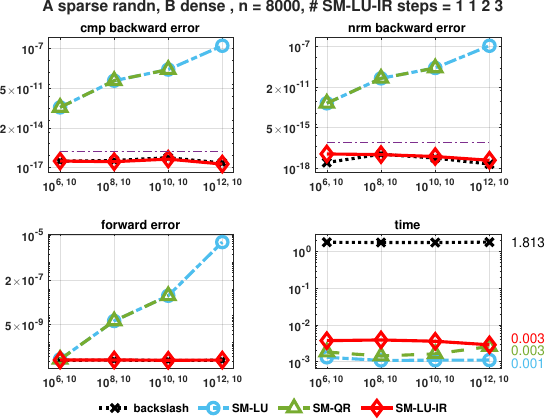}
\caption{Results for Example~\ref{ill_ill_smallNorm_sparse:ex} with sparse $A$. Case 1(i): $\kappa(A), \kappa(A+uv^T) \gg 1$, small-norm solution. 
}
    \label{fig:ill_moderate_smallNorm_sparse}
 \end{figure}

\begin{figure}[!h]
\center
\includegraphics[width=0.7\textwidth]{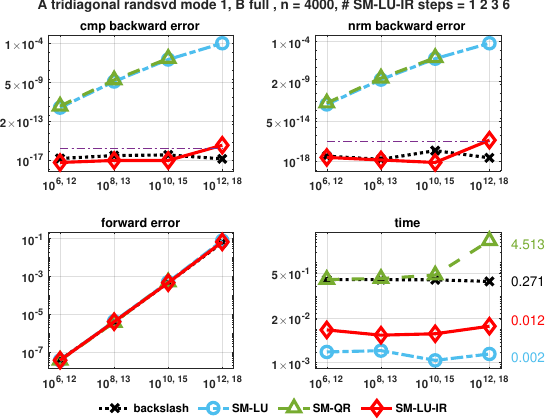}
\caption{Results for Example~\ref{ill_ill_smallNorm:ex}. Case 1(i): $\kappa(A), \kappa(A+uv^T) \gg 1$, small-norm solution. }
    \label{fig:ill_ill_smallNorm_tridiag}
 \end{figure}

As the condition number of $A$ grows, the number of IR steps increases from $1$ to $3$ for the most ill-conditioned matrix. 
The number of steps are also reported in the title of Figure~\ref{fig:ill_moderate_smallNorm_sparse}. To be clear, for the first test where $\kappa(A) = 10^6$ and $\kappa(B) = 5.1 \times 10^{9} \approx 10^{10}$, \texttt{SM-LU-IR} needs one IR step, and in the last test which is the most ill-conditioned example, it requires three IR steps. Note that \texttt{SM-LU-IR} exhibits backward stability while the additional IR iterations required for the more ill-conditioned tests incur no significant time penalty. Also, as $A$ is sparse, all variants of SM can take advantage of structure to speed up their computation. $B$ is not sparse, which is why backslash applied to $B$ is the slowest method in all four tests. 

We also note that \texttt{SM-QR} does not perform well for the most ill-conditioned example for which the backward and forward errors are NaN. This is due to infinity entries in $y$ and $z$ computed with the QR factors when applied to the sparse matrix $A$. The loss of accuracy could be mitigated at the cost of slowing down \texttt{SM-QR} if $A$ is first converted to a full matrix, thereby neglecting its sparsity. Even then, as with the previous three tests, \texttt{SM-QR} without IR would still fail to achieve backward stability.
\end{example}

\begin{example} \label{ill_ill_smallNorm:ex} \normalfont
We generate $A$ of size $n = 4000$ using \texttt{randsvd} matrices of mode 1 in the MATLAB gallery with condition numbers of $10^6, 10^8, 10^{10}$ and $10^{11}$. We also generate $u$ and $v$ as \texttt{randn} vectors. The condition number of matrices $B=A+uv^T$ are $10^{11}, 10^{13}, 10^{15}$ and $10^{18}$, respectively; all of these are reported on the horizontal axes in Figure~\ref{fig:ill_ill_smallNorm_tridiag}. We take the exact solution to also be a \texttt{randn} vector for which $\|x\| \approx 30$ and then form $b:= Bx$.

The backward error of SM-LU-IR is comparable with backslash applied to $B$, but it is faster even in the last case where 6 IR steps are required.
\end{example}

\subsection{Case 1(ii): both $\mathbf{\kappa(A), \kappa(A+uv^T) \gg 1}$, large-norm solution}

\begin{example} \label{ill_ill_largeNorm:ex} \normalfont
SM performs stably here. See Figure~\ref{fig:ill_ill_largeNorm}. Here we set $A$ to be tridiagonal as it is easier to control $\kappa(B)$. In addition, $b$ is set to be a \texttt{randn} vector and the exact solution is considered the one obtained with backslash.

 \begin{figure}[!b]
\center
\includegraphics[width=0.7\textwidth]{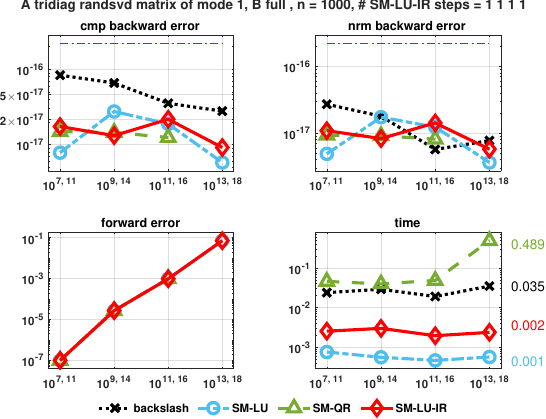}
\caption{Results for Example~\ref{ill_ill_largeNorm:ex}. Case 1(ii): $\kappa(A), \kappa(A+uv^T) \gg 1$, large-norm solution.}
    \label{fig:ill_ill_largeNorm}
 \end{figure}
 
\end{example}

\subsection{Case 2(i): $\mathbf{\kappa(A) = \mathcal{O}(1), \kappa(A+uv^T) \gg 1}$, small-norm solution}

\begin{example} \label{well_ill_smallNorm:ex} \normalfont
  We generate $A$, tridiagonal of size $n = 1000$ using \texttt{randsvd} matrices of mode 5 in the MATLAB gallery with condition numbers of $10^1, 10^2, 10^{3}$ and $10^{4}$. We also generate $u$ and $v$ as \texttt{randn} vectors. The condition number of matrices $B=A+uv^T$ are $10^{5}, 10^{6}, 10^{7}$ and $10^{8}$, respectively. We take the exact solution to also be a \texttt{randn} vector for which $\|x\| \approx 30$ and then form $b:= Bx$. See Figure~\ref{fig:well_ill_smallNorm} suggesting instability of SM at least to some extent. 

\begin{figure}
\center
\includegraphics[width=0.7\textwidth]{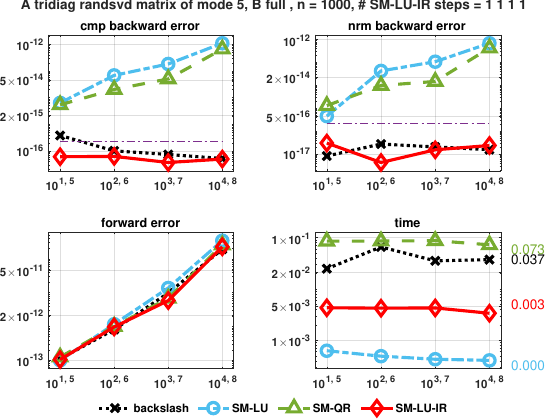}
\caption{Results for Example~\ref{well_ill_smallNorm:ex}. Case 2(i): $\kappa(A) = \mathcal{O}(1), \kappa(A+uv^T) \gg 1$, small-norm solution}
    \label{fig:well_ill_smallNorm}
 \end{figure}
 
\end{example}

\subsection{Case 2(ii): $\mathbf{\kappa(A) = \mathcal{O}(1), \kappa(A+uv^T) \gg 1}$, large-norm solution}

\begin{example} \label{well_ill_largeNorm:ex} \normalfont
In this problem where solution has a large-norm, again SM performs stably. See Figure~\ref{fig:well_ill_largeNorm}.
\begin{figure}
\center
\includegraphics[width=0.7\textwidth]{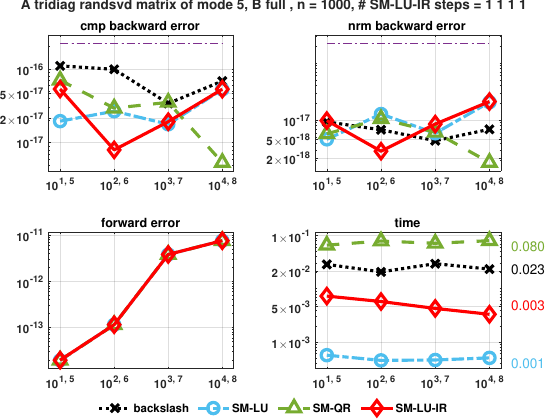}
\caption{Results for Example~\ref{well_ill_largeNorm:ex}. Case 2(ii): $\kappa(A) = \mathcal{O}(1), \kappa(A+uv^T) \gg 1$, large-norm solution.}
    \label{fig:well_ill_largeNorm}
 \end{figure}
 
\end{example}

\subsection{Case 3: $\mathbf{\kappa(A) \gg 1, \kappa(A+uv^T)  = \mathcal{O}(1)}$, small-norm solution}

\begin{example} \label{ill_well_smallNorm:ex} \normalfont
We generate $A$, pentadiagonal of size $n = 1000$ using \texttt{randsvd} matrices of mode 2 in the MATLAB gallery with condition numbers of $\kappa = 10^7, 10^9, 10^{11}$ and $10^{13}$. Matrices $A$ of this type have one small singular value that is equal to $\kappa^{-1}$ and all the remaining $n-1$ singular values are equal to 1. In order to make $B$ well-conditioned, we take $u$ to be a random multiple of the right singular vector of $A$ corresponding to its smallest singular value and analogously for $v$. Therefore the smallest singular value of $A$ is replaced with one of the same order as the rest of its singular values drastically improving the condition number of $B = A+ uv^T$ to $\mathcal{O}(1)$. 

We take $x$ a random vector and set $b = Bx$. Therefore, both $x$ and $b$ have a small norm. See Figure~\ref{fig:ill_well_smallNorm}. Interestingly, SM is unstable once again; reflecting the discussion in Section~\ref{sec:source}. Perhaps even more surprising is that not only SM-IR behaves in a backward stable manner, even the forward error is $\mathcal{O}(\eps_M)$ although this is a fixed-precision IR! Given that the plain SM is not backward stable, it has no chance of getting close to SM-IR in terms of forward stability, even though $B$ is well-conditioned. In such cases, if SM is required (for speed), it is strongly advisable to use IR.

\begin{figure}[!h]
\center
\includegraphics[width=0.7\textwidth]{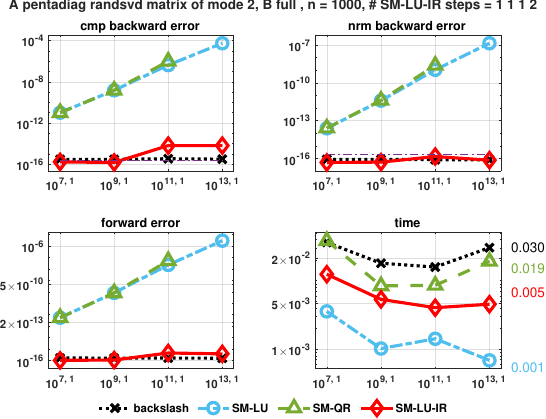}
\caption{Results for Example~\ref{ill_well_smallNorm:ex}. Case 3(i): $\mathbf{\kappa(A) \gg 1, \kappa(A+uv^T)  = \mathcal{O}(1)}$, small-norm solution}
    \label{fig:ill_well_smallNorm}
 \end{figure}
\end{example}

\begin{figure}[!h]
\center
\includegraphics[width=0.7\textwidth]{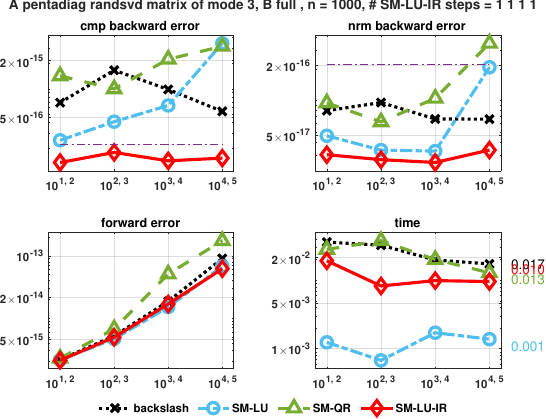}
\caption{Results for Example~\ref{well_well_smallNorm:ex}. Case 4: both $\mathbf{\kappa(A), \kappa(A+uv^T)  = \mathcal{O}(1)}$, small-norm solution}
    \label{fig:well_well_smallNorm}
 \end{figure}

\subsection{Case 4: both $\mathbf{\kappa(A), \kappa(A+uv^T)  = \mathcal{O}(1)}$, small-norm solution}

\begin{example} \label{well_well_smallNorm:ex} \normalfont
We generate pentadiagonal $A$ of size $n = 1000$ using \texttt{randsvd} matrices of default mode 3 in the MATLAB gallery with condition numbers of $10^1, 10^2, 10^{3}$ and $10^{4}$. 

We take $u$ and $v$ as random vectors, but also normalized to ensure that $\kappa(A+uv^T)$ does not deviate significantly from the scenario under consideration. The condition number of matrices $A$ and $B=A+uv^T$ can be seen on the horizontal axes of Figure~\ref{fig:well_well_smallNorm}. We take the exact solution to also be a \texttt{randn} vector for which $\|x\| \approx 30$ and then form $b:= Bx$ with $\|b\| \approx 10$. This is a situation where SM should be stable as discussed in the introduction; and the backward errors are all small. It is still worth noting that the backward error is observed to be improved by IR to $\mathcal{O}(\eps_M)$ from $\mathcal{O}(\kappa(A)\eps_M)$.
\end{example}

As in this case both $A$ and $A+uv^T$ are well-conditioned, with a random $b$ the solution cannot have a large norm. So, we skip the other sub-case again.

\subsection{Effect of IR}
Above we have shown the results of SM-IR after the final IR step. It is interesting to see how covergence is taking place as the IR steps proceed. We present these in the figures below, for Case 1(i) where we required the most number of IR iterations. We observe (as did in all cases with varying speed of improvement by IR) that both the residual and backward error converge steadily, until the backward error reaches $O(\eps_M)$.

% Case 1
\begin{figure}[!h]
\center
\includegraphics[width=0.4\textwidth]{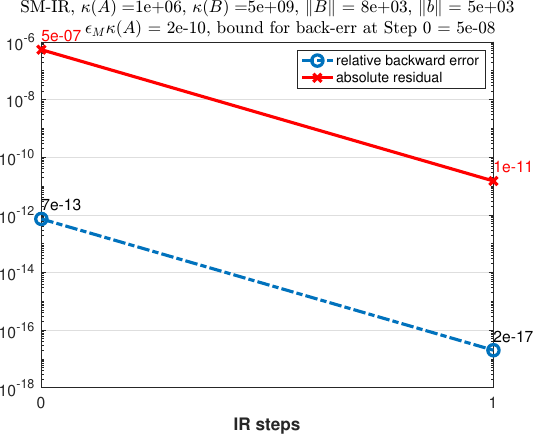}
\includegraphics[width=0.4\textwidth]{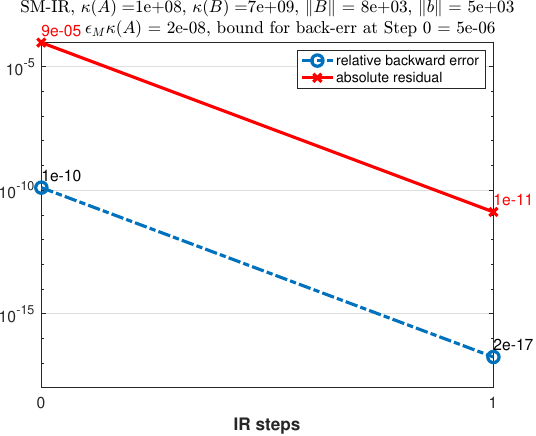}
\includegraphics[width=0.4\textwidth]{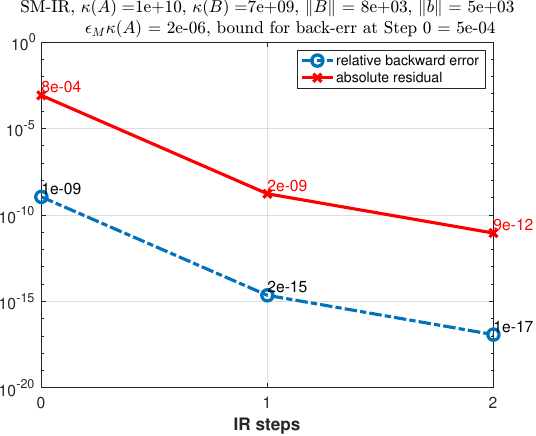}
\includegraphics[width=0.4\textwidth]{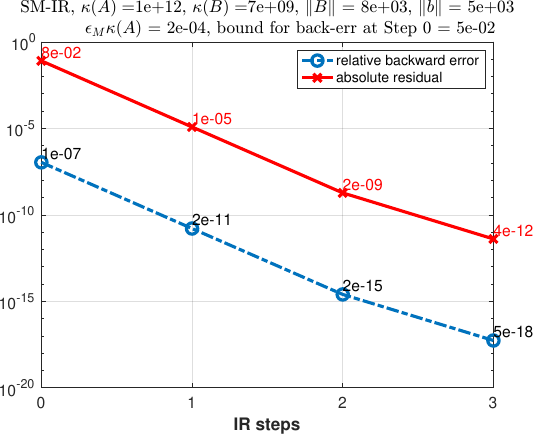}
\caption{Results for Example~\ref{ill_ill_smallNorm_sparse:ex}. $A$ sparse, $n = 8000$. Case 1(i): $\kappa(A), \kappa(A+uv^T) \gg 1$, small-norm solution}
    \label{fig:IR_ill_moderate_smallNorm_sparse}
 \end{figure}
\begin{figure}[!h]
\center
\includegraphics[width=0.4\textwidth]{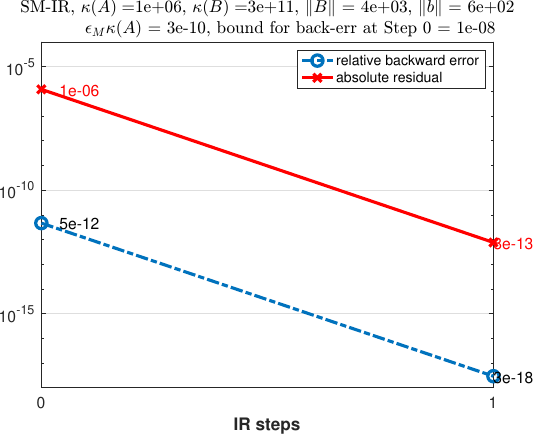}
\includegraphics[width=0.4\textwidth]{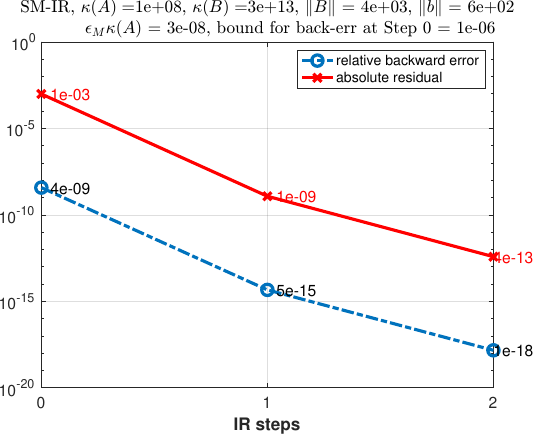}
\includegraphics[width=0.4\textwidth]{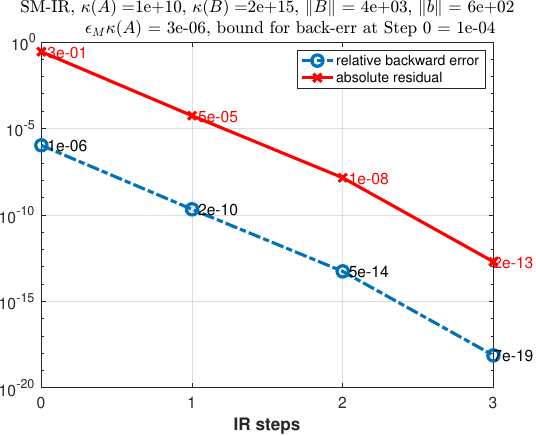}
\includegraphics[width=0.4\textwidth]{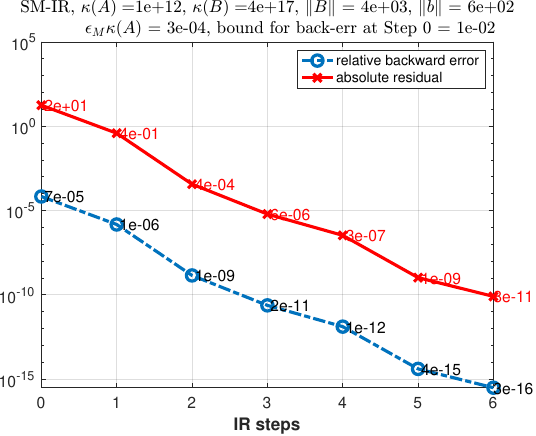}
\caption{Results for Example~\ref{ill_ill_smallNorm:ex}, $A$ tridiagonal, $n = 4000$. Case 1(i): $\kappa(A), \kappa(A+uv^T) \gg 1$, small-norm solution.}
    \label{fig:IR_ill_ill_smallNorm}
 \end{figure}

\paragraph{Summary of experiments}
In all our experiments, SM(-LU)-IR was able to compute a backward stable solution after $\mathcal{O}(1)$ IR steps, without sacrificing speed. Based on these promising empirical results, we conjecture that SM-IR is backward stable when applied to numerically nonsingular linear systems. As noted in the introduction, our theory requires the assumptions to hold, so is not sufficient to prove this in full. A full proof is left as an open problem. 

\subsection*{Acknowledgments}
We are grateful to David Bindel for pointing out the connection to bordered linear systems and for bringing~\cite{Govaerts91} to our attention. We also thank Alex Townsend for highlighting the relevance of structured-plus-low-rank linear systems in ultraspherical spectral methods.

%\bibliographystyle{siam}
%\bibliography{refs}
\bibliography{manuscript} % manuscript.bib (no .bib in the command), just for arXiv

\newpage
\appendix

\section{Govaerts' work}
The starting point is to introduce the scalar variable $\zeta = v^T x$, which serves as a convenient intermediate step in the computation. We can now regard~\eqnref{eq:maingoal} as equivalent to the bordered linear system 
\begin{equation}
\label{eq:bordered}
\begin{bmatrix}
A & u \\
v^T & -1
\end{bmatrix}
\begin{bmatrix}
x \\
\zeta
\end{bmatrix} =
\begin{bmatrix}
b \\
0
\end{bmatrix}.
\end{equation}
Govaerts~\cite{Govaerts91} proposed algorithms for solving general $(n+1) \times (n+1)$ bordered linear systems, where the $(1,1)$ block is a matrix $A$ for which a black-box solver is assumed to be available. He discussed variants of block Gaussian elimination, both without and with IR. For instance, the BEC (block elimination Crout) algorithm of~\cite{Govaerts91}, when applied to~\eqnref{eq:bordered}, employs the following LU decomposition
\[
\begin{bmatrix}
A & u \\
v^T & -1
\end{bmatrix}
=
\begin{bmatrix}
A & 0 \\
v^T & -1-v^T A^{-1}u
\end{bmatrix}
\begin{bmatrix}
I_n & A^{-1} u \\
0 & 1
\end{bmatrix}.
\]
Backward and forward substitutions then give the SM formula. The following mixed forward-backward error bound is how Govaerts's analysis of BEC~\cite[Prop. 3.3]{Govaerts91} adapts to~\eqnref{eq:bordered}.
\begin{proposition}
Let $S$ be a normwise backward stable solver for $A$ with stability constant $C_S$\footnote{i.e., when $S$ is applied for solving $A y = b$ in floating-point arithmetic, then there exist a modest constant $C_S$, a matrix $\Delta A$ and a vector $\delta b$ such that $(A + \Delta A) \hat y = b + \delta b$ with $\| \Delta A\|_2 \leq C_S \eps_M \|A\|_2$ and $\| \delta b\|_2 \leq C_S \eps_M \|b\|_2$.} Then, $\hat x, \hat \zeta$ satisfy
\begin{equation}
\label{eq:pert_bordered}
\begin{bmatrix}
A+ \Delta A & u + \delta u\\
(v + \delta v)^T & -1+ \delta
\end{bmatrix}
\begin{bmatrix}
\hat x\\
\hat \zeta
\end{bmatrix} =
\begin{bmatrix}
b + \delta b\\
0
\end{bmatrix}
+ 
\begin{bmatrix}
T\\
q^T
\end{bmatrix} \hat y
\end{equation}
where 
\begin{align*}
 &
\| \Delta A\|  \leq (2+C_S) \eps_M \|A\| + \mathcal{O}(\eps_M^2),\\
 &\| \delta u\| \leq C_S \eps_M \|u\|,\\
 & 
\| \delta v\|  \leq (5+C_{IP}) \eps_M \|v\| + \mathcal{O}(\eps_M^2),\\
 &|\delta|  \leq 3 \eps_M \exp(3\eps_M),\\
 &\| \delta b\|  \leq C_{S} \eps_M  \|b\|,\\
 & 
\| T\| \leq (2C_S + 1)\ \eps_M \|A\| + \mathcal{O}(\eps_M^2),\\
 & 
\| q\|  \leq (4 + 2C_{IP}) \eps_M \|v\| + \mathcal{O}(\eps_M^2),
\end{align*}
in which $C_{IP} \leq \frac{n}{1-n \eps_M}$ is the constant appearing in the backward error for the inner product of two vectors. 
\end{proposition}
\ \\
Govaerts~\cite{Govaerts91} also discusses two other algorithms (BED and BEM) which, when applied to~\eqnref{eq:bordered} are not equivalent to the SM formula as it is commonly implemented, for example in Alg~\ref{SM:alg} or in \cite[p. 487]{Higham02}. Directions for future work include investigating the potential implications of these alternative algorithms for the SM(W) framework.

\section{Related backward error bounds} \label{backward_err_25:sec}
Although this paper is {\em not} concerned with the stability of the SMW formula for matrix inversion, in this subsection, we discuss how the backward error bounds derived by Ma, Boutsikas, Ghadiri and Drineas \cite{Ma25} for a computed inverse of $B$ using SMW relate to Higham's question. Although it is common knowledge that multiplying by an explicit inverse is not an advisable way of solving a linear system in practice, one might argue that, at-least in theory, SMW could first be applied to invert $A+UV^T$, followed by a multiplication with $b$. Could the existing error bounds for matrix inversion via SMW then lead to alternative solutions to Higham's question? Here, we present an attempt in this direction that yields a negative answer. We first outline backward error bounds derived in \cite{Ma25} for the low-rank perturbed matrix inversion and then examine their implications when combined with ideal forward error bounds.
\begin{proposition} \label{backward_thm_25}
\normalfont 
\cite[Thm. 6]{Ma25} 
Let $\lambda := \|U\|_2 \|V\|_2$ 
and 
\begin{align}
\epsilon_1 &:= \|\tilde{A}^{-1} - A^{-1} \|_2, \label{eqn:assumption1} \\
\epsilon_2 & := \|Z^{-1} - (I + V^T \tilde{A}^{-1} U)^{-1} \|_2 \label{eqn:assumption2}
\end{align}
where $Z$ is a computed approximate inverse of $I + V^T \tilde{A}^{-1} U$. If 
 \begin{align}
      & \| I + V^T A^{-1} U\|_2 \leq \beta, \nonumber\\ %\label{beta_assumption}\\
      & \widetilde{A}\ \text{and}\ (Z^{-1})^{-1} - V^T \widetilde{A}^{-1} U\ \text{are invertible},\nonumber\\
      %\label{backward_thm_assumption4}\\
      & \epsilon_1 < \frac{1}{2 \|A\|_2},  \label{backward_thm_assumption1}\\
      & \epsilon_2 < \frac{1}{2 \left( \beta+ \lambda \, \epsilon_1  \right)},  \nonumber \\ %\label{epsilon2_small\\
      &  2\left( \beta+ \lambda \, \epsilon_1  \right)^2 \epsilon_2 < \frac{1}{2}. \nonumber%\label{epsilon2_small_2}.
\end{align}
Then
    \begin{align}
    \| B -\left(\widetilde{A}^{-1} - \widetilde{A}^{-1} U Z^{-1} V^T \widetilde{A}^{-1}\right)^{-1}\|_2 
    &\leq 2 \epsilon_1 \|A\|_2^2 + 4 \lambda\epsilon_2 \left(\beta+ \lambda \, \epsilon_1  \right)^2. \label{backward_fullbound}
\end{align}
\end{proposition}

Note that $\epsilon_1$ and $\epsilon_2$ are absolute forward errors in computing the inverse of $A$ and the inverse of $I + V^T A^{-1} U$, respectively. In particular, $\epsilon_2$ assumes that the addition of matrices and matrix-matrix multiplications in $I + V^T \tilde{A}^{-1} U$ are performed in exact arithmetic such that $\epsilon_2$ solely reflects the forward error in the operation of inverting $I + V^T \tilde{A}^{-1} U$. Assuming further that the update $UV^T$ has a small norm relative to that of $A$, the paper \cite{Ma25} proves that the backward error~\eqnref{backward_fullbound} simplifies to the following bound.
\begin{proposition} \label{backward_corr_25}
\normalfont \cite[Cor. 7]{Ma25}  
Assume that $\epsilon_1$, $\epsilon_2$  and $\sigma_{\min}(A)$ are all upper bounded by one and also assume that $\epsilon_1$ and $\epsilon_2$ satisfy the assumptions of Proposition~\ref{backward_thm_25}. Additionally, assume that $\lambda \leq \frac{\sigma_{\min}(A)}{2}$. Then, the right-hand side of the bound of eqn.~(\ref{backward_fullbound}) simplifies to:
\begin{equation*}
\| B -\left(\widetilde{A}^{-1} - \widetilde{A}^{-1} U Z^{-1} V^T \widetilde{A}^{-1}\right)^{-1}\|_2 \leq 2 \epsilon_1 \|A\|_2^2 + 8\epsilon_2.
\end{equation*}   
\end{proposition}

Next, let us recall the {\em ideal} forward error bound for matrix inversion. This follows the derivation in \cite[p. 261]{Higham02} where we replace componentwise perturbations with those in the 2-norm. Suppose $A$ is perturbed to $A + \Delta A$ with $\|\Delta A\|_2 \leq \eps \|A\|_2$. If $\tilde{A}^{-1}$, an approximate inverse of $A$, is the exact inverse of a perturbation of $A$, i.e., if it satisfies $\tilde{A}^{-1} = (A+ \Delta A)^{-1}$, then $(A+ \Delta A) \tilde{A}^{-1} = \tilde{A}^{-1} (A+ \Delta A) = I$ and so
 \[
\tilde{A}^{-1} = (A+ \Delta A)^{-1} = A^{-1} - A^{-1} \ \Delta A \ A^{-1} + \mathcal{O}(\eps^2)
 \]
 which gives the following standard forward error bound
\[
\|\tilde{A}^{-1} - A^{-1} \|_2 \leq  \| A^{-1} \ \Delta A\|_2 \ \|A^{-1} \|_2 + \mathcal{O}(\eps^2)
\leq \eps \kappa_2(A) \| A^{-1} \|_2 + \mathcal{O}(\eps^2).
\]
Arguing similarly for the inverse of $I + V^T \tilde{A}^{-1} U$ and replacing $\eps$ with machine epsilon $\eps_M$, we can replace $\epsilon_1, \epsilon_2$ in \eqnref{eqn:assumption1} and \eqnref{eqn:assumption2} by
\begin{align}
\epsilon_1 &\leq \eps_M \kappa_2(A) \| A^{-1} \|_2 + \mathcal{O}(\eps_M^2) \label{eqn:assumption1new} \\
\epsilon_2 & \leq \eps_M \kappa_2(Z) \|Z^{-1} \|_2 + \mathcal{O}(\eps_M^2)\label{eqn:assumption2new} 
\end{align}
which are typically less than one, but could be even larger if $A$ and $Z$ are ill-conditioned. Since 
\eqnref{eqn:assumption1new} and \eqnref{eqn:assumption2new} are ideal bounds, in view of the assumption \eqnref{backward_thm_assumption1}, we can see that in practice Propositions~\ref{backward_thm_25} and ~\ref{backward_corr_25} are applicable to matrices $A$ for which 
\[
\kappa_2(A) \leq (2 \eps_M)^{-1/2}.
\]
In addition, incorporating \eqnref{eqn:assumption1new} and \eqnref{eqn:assumption2new} into Proposition 
~\ref{backward_corr_25} yields the following backward error bound
\begin{equation*}
\| B -\left(\widetilde{A}^{-1} - \widetilde{A}^{-1} U Z^{-1} V^T \widetilde{A}^{-1}\right)^{-1}\|_2 \leq \eps_M 
\Big( 2 \kappa_2(A)^2\|A\|_2 + 8 \kappa_2(Z)^2 \|Z^{-1}\|_2 \Big) + \mathcal{O}(\eps_M^2).
\end{equation*}   
Applying the reverse triangle inequality to $B = A + UV^T$, the definition of $\lambda$ and the bound on it imposed in Proposition~\ref{backward_corr_25}, it is easy to cast the above absolute bound into the following relative backward error bound 
\begin{equation*}
\frac{\| B -\left(\widetilde{A}^{-1} - \widetilde{A}^{-1} U Z^{-1} V^T \widetilde{A}^{-1}\right)^{-1}\|_2}{\|B\|_2} \leq \eps_M \frac{4 \kappa_2(A)^3\|A\|_2 + 16 \kappa_2(Z)^2 \|Z^{-1}\|_2 \|A^{-1}\|_2}{2\kappa_2(A) - 1} + \mathcal{O}(\eps_M^2).
\end{equation*}   
where the bound is, once more, essentially proportional to the square of the condition number of $A$.
\end{document}